\documentclass[11pt]{amsart}

\usepackage{amsmath, amssymb, amscd, cancel, graphicx, soul, stmaryrd}
\usepackage[T1]{fontenc}
\usepackage{mathdots}

\usepackage[dvipsnames,usenames]{color}
\usepackage[colorlinks=true, urlcolor=NavyBlue, linkcolor=NavyBlue, citecolor=NavyBlue]{hyperref}
\usepackage{float}
\usepackage{amsmath, amsthm, amssymb}
\usepackage{amsfonts}
\usepackage{enumerate}
\usepackage{paralist}
\usepackage{xypic}
\usepackage{subfigure}
\usepackage{hyperref}
\usepackage{color}
\usepackage{marginnote}
\hypersetup{
linkbordercolor={1 0 0}, 
citebordercolor={0 1 0} 
}
\newtheorem{thm}{Theorem}[section]

\newtheorem{cor}[thm]{Corollary}
\newtheorem{lem}[thm]{Lemma}

\newtheorem{defin}[thm]{Definition}
\newtheorem{rmk}[thm]{Remark}

\def\s{{\mathfrak s}}

\def\bQ{{\mathbb Q}}

\def\bZ{{\mathbb Z}}

\def\sf{{\mathrm{Sf}}}

\DeclareMathOperator{\tb}{tb}
\DeclareMathOperator{\rot}{rot}
\newcommand{\maxtb}{\overline{\tb}}
\newcommand{\cbar}{\overline{c}}
\newcommand{\zee}{{\mathbb Z}}
\newcommand{\N}{{\mathbb N}}
\newcommand{\F}{{\mathbb F}}
\newcommand{\X}{{\mathbb X}}
\renewcommand{\L}{{\mathcal L}}
\newcommand{\hfhat}{\widehat{HF}}
\newcommand{\cfhat}{\widehat{CF}}
\newcommand{\spinc}{{\mbox{spin$^c$} }}

\newcommand{\KH}{HKh^*}
\newcommand{\eff}{{\mathcal F}}

\begin{document}

\title[On the Stein framing number of a knot]%
{On the Stein framing number of a knot}

\author{Thomas E. Mark}
\address{Department of Mathematics, University of Virginia, Charlottesville, VA 22904}
\email{tmark@virginia.edu}

\author{Lisa Piccirillo}
\address{Department of Mathematics, University of Texas, Austin, TX 78712}
\email{lpiccirillo@math.utexas.edu}

\author{Faramarz Vafaee}
\address{Department of Mathematics, California Institute of Technology \\ Pasadena, CA 91125}
\email{vafaee@caltech.edu}

\maketitle

\medskip

\medskip


\begin{abstract}For an integer $n$, write $X_n(K)$ for the 4-manifold obtained by attaching a 2-handle to the 4-ball along the knot $K\subset S^3$ with framing $n$. It is known that if $n< \maxtb(K)$, then $X_n(K)$ admits the structure of a Stein domain, and moreover the adjunction inequality implies there is an upper bound on the value of $n$ such that $X_n(K)$ is Stein. We provide examples of knots $K$ and integers $n\geq \maxtb(K)$ for which $X_n(K)$ is Stein, answering an open question in the field. In fact, our family of examples shows that the largest framing such that the manifold $X_n(K)$ admits a Stein structure can be arbitrarily larger than $\maxtb(K)$. We also provide an upper bound on the Stein framings for $K$ that is typically stronger than that coming from the adjunction inequality.
\end{abstract}

\section{Introduction}

A differential topological characterization of smooth manifolds that admit the structure of Stein manifolds has been known for many years, dating to the seminal work of Eliashberg~\cite{eliash}. For a (real) four-dimensional manifold $X$, there is a Stein structure on $X$ if and only if $X$ admits a handle decomposition containing only handles of index 0, 1, and 2, such that the attaching circles of the 2-handles satisfy a framing condition. Here, and throughout, we consider {\it compact} $X$ and by ``Stein structure'' on $X$ we mean the structure of a Stein domain as described in \cite[Chapter 8]{OzbagciStipsicz}, for example. To describe the framing condition, note that the 1-skeleton of such $X$ is diffeomorphic to a boundary sum of copies of $S^1\times D^3$, which admits a unique Stein structure. In particular the boundary of the 1-skeleton is a connected sum of copies of $S^1\times S^2$ with the contact structure induced by the Stein structure (consisting of the field of complex lines in the tangent bundle). In the case that there are no 1-handles we mean the ``empty'' connected sum: $S^3$, bounding the Stein 0-handle $D^4$. The condition on the 2-handles of $X$ is that they be attached along Legendrian curves (i.e., curves everywhere tangent to the contact structure), with framing differing from that induced by the contact structure by a single negative twist.

If one is given a handle decomposition on a smooth 4-manifold $X$, it is not always a simple matter to decide if the handle decomposition can be modified to fit Eliashberg's criteria. Our aim here is to illustrate this point in one of the homotopically simplest cases: that of a smooth 4-manifold obtained by attaching a single 2-handle along a knot $K\subset S^3$. Recall that a framing of a knot in $S^3$ can be invariantly described by an integer representing the difference between the given framing and the framing induced by a Seifert surface; if $\mathcal L$ is a Legendrian then the contact framing is usually called the Thurston-Bennequin number of $\L$, written $\tb(\L)\in \zee$. Any smooth knot is isotopic to many Legendrian knots, with varying contact framings, but a basic result of Bennequin \cite{bennequin} implies that for a given smooth knot $K$, there is an upper bound for $\tb(\L)$ for any Legendrian $\L$ isotopic to $K$. We write the maximum Thurston-Bennequin number of all Legendrian representatives of $K$ as $\maxtb(K)$. 

For an integer $n$, write $X_n(K)$ for the 4-manifold obtained by attaching a 2-handle to the 4-ball along $K$ with framing $n$. From Eliashberg's criterion, if $n< \maxtb(K)$, then $X_n(K)$ admits the structure of a Stein domain (indeed, for any $t\leq \maxtb(K)$, there is a Legendrian representative of $K$ with Thurston-Bennequin number $t$). The question we will address, stated explicitly in \cite{Yas15b} for example, is whether $X_n(K)$ admits a Stein structure {\it only} when $n<\maxtb(K)$. We introduce the following terminology.

\begin{defin} For a smooth knot $K\subset S^3$, the {\em Stein framing number} of $K$, written $\sf(K)$, is the largest framing $n$ such that the manifold $X_n(K)$ admits a Stein structure.
\end{defin}

By the remarks above, one knows $\maxtb(K)-1\le\sf(K)$. In the other direction, the adjunction inequality for Stein manifolds \cite{LM97,FSimmersedthom,OSthom} shows that $\sf(K)\leq 2g_*(K)-2$ where $g_*(K)$ is the minimal genus of a proper smoothly embedded orientable surface in $D^4$ with boundary $K$. In some cases these inequalities determine $\sf(K)$: as mentioned in \cite{Yas15b}, there are many knots---such as positive torus knots---that admit Legendrian representatives whose Thurston-Bennequin number equals $2g_*-1$, proving that for these examples $\sf = \maxtb-1$. 

Our main results provide on the one hand a more refined upper bound for $\sf(K)$, and on the other hand a family of examples demonstrating that the Stein framing number can be arbitrarily larger than the maximum Thurston-Bennequin number. These are the first examples of knots $K$ for which $X_n(K)$ is shown to be Stein for some $n\geq \maxtb(K)$; in particular we answer Problem 1.3 of \cite{Yas15b} negatively:

\begin{thm} For any integer $m\geq 0$, there exists a knot $J_m\subset S^3$ such that $\sf(J_m) \geq \maxtb(J_m) +m$.
\label{Thm:largegaps}
\end{thm}

To state the upper bound for $\sf$, let $K$ be a knot in $S^3$ and $[\Sigma]$ a generator of $H_2(X_n(K);\zee)$, for example the generator obtained by capping off a Seifert surface for $K$. Let
\[
\cbar(K) = \max\{|\langle c_1(J), [\Sigma]\rangle|, \mbox{ $J$ a Stein structure on $X_n(K)$ with $n = \sf(K)$}\}.
\]

\begin{thm}\label{Thm:upperbd} For a knot $K\subset S^3$, the Stein framing number satisfies
\begin{equation}\label{weakbd}
\sf(K) + \cbar(K) \leq 2\tau(K),
\end{equation}
where $\tau(K)\in \zee$ is the concordance invariant arising from knot Floer homology defined by Ozsv\'ath-Szab\'o \cite{OSfourball} and Rasmussen \cite{RasThesis}.
If $\epsilon(K) = 1$, where $\epsilon\in \{-1,0,1\}$ is the invariant defined by Hom in \cite{Hom1}, then
\begin{equation}\label{strongbd}
\sf(K) + \cbar(K) \leq 2\tau(K)-2.
\end{equation}
\end{thm}

The ideas for the proofs are as follows. For Theorem \ref{Thm:largegaps}, we make use of work of Osoinach \cite{Oso06}, extended by Abe-Jong-Luecke-Osoinach \cite{AJLO14}, which gives a method to produce, for $m\in \N$, pairs of distinct knots $P_{m}$, $Q_{m}$ such that $X_{-m}(P_m) \cong X_{-m}(Q_m)$. If $X_{-m}$ denotes this common 4-manifold, then $X_{-m}$ is Stein whenever $-m$ is less than the maximal Thurston-Bennequin number of either $P_m$ or $Q_m$. The main work in the proof of Theorem \ref{Thm:largegaps} is in estimating these maximal Thurston-Bennequin numbers, and in particular we show that for our examples $\maxtb(P_m) =  -m+1$ while $\maxtb(Q_m)\leq -2m+3$. It follows that $X_{-m}$ is Stein, but the framing coefficient $-m$ can be made arbitrarily larger than $\maxtb(Q_m)$. The required estimates on $\maxtb$ are derived from Khovanov homology, using in particular a theorem of Ng \cite{Ng05}. This proof occupies Section \ref{Sec:largegaps}.

Theorem \ref{Thm:upperbd} follows from observing that a Stein cobordism between 3-manifolds induces a nontrivial homomorphism in Heegaard Floer homology, and using the techniques available from knot Floer theory to constrain the framings for which such a homomorphism is possible. The details are carried out in Section \ref{Sec:upperbd}.

\subsection*{Further Remarks and Questions} 

The proof of Theorem \ref{Thm:largegaps} gives examples of knots $J$ and an individual framing $m'\gg\maxtb(J)$ such that $X_{m'}(J)$ is Stein. As remarked previously, it is also true that $X_n(J)$ is Stein for any $n<\maxtb(J)$. It is not obvious, however, whether $X_n(J)$ is Stein when $\maxtb(J)\leq n < m'$. For a given knot $K$ one might ask whether the set of framings $n$ such that $X_n(K)$ is Stein can contain ``gaps'' of this sort, or whether $X_n(K)$ is Stein for every $n\leq \sf(K)$.

One might also ask whether the stronger bound \eqref{strongbd} in Theorem \ref{Thm:upperbd} always holds, or whether there exist examples of knots realizing equality in \eqref{weakbd}. Work of Plamenevskaya \cite{Olga2004} shows that for any Legendrian knot $K$ in $S^3$ one knows $\tb(K) + |\rot(K)| \leq 2\tau(K) -1$. Since $\tb(K)-1$ is a framing for which the trace of the $(\tb(K)-1)$-surgery admits a Stein structure, and in this cobordism the corresponding Chern number is exactly $\rot(K)$, we see that for  the extreme case $\sf(K) = \maxtb(K) - 1$, the inequality \eqref{strongbd} is true without the assumption on $\epsilon(K)$.  However, with our methods the two cases in the theorem cannot be avoided, in the sense that for knots with $\epsilon = 0$ or $-1$ one can always find a framing and a \spinc structure on the corresponding surgery cobordism inducing a nontrivial map in Floer homology, such that the sum of the framing and the Chern number is equal to $2\tau$.

Finally, we remark that in many cases Theorem \ref{Thm:upperbd} refines the upper bound on $\sf(K)$ given by the adjunction inequality for Stein manifolds. The adjunction inequality implies that if $X_n(K)$ is Stein with first Chern class $c$, then $n + |\langle c, \Sigma\rangle| \leq 2g_*(K)-2$.  When $\epsilon(K) =1$ it is clear that Theorem \ref{Thm:upperbd} improves on this from the fact that $|\tau(K)|\leq g_*(K)$ \cite[Corollary 1.3]{OSfourball} (strictly, the improvement comes via Theorem \ref{Thm:constraint} and Corollary \ref{Cor:constraint} below, applied to the given $n$ and $c$). If $\epsilon(K) = -1$ and $\tau(K)<0$, then $g_*(K) >0$, so the right-hand side of \eqref{weakbd} is negative but the adjunction inequality gives a nonnegative bound for $\sf(K)$. If $\epsilon(K) = -1$ and $\tau(K)\ge 0$, then Corollary 4 of \cite{Hom1} shows $\tau(K)\leq g_*(K) - 1$, so that \eqref{weakbd} is at least as strong as the adjunction bound. If $\epsilon(K) = 0$ then $\tau(K) = 0$ (c.f. \cite{Hom1}), so \eqref{weakbd} is at least as good as adjunction unless $g_*(K) = 0$, i.e., unless $K$ is slice. 
\subsection*{Acknowledgements} We would like to thank John Etnyre, David Gay and John Luecke for helpful discussions and their interest in our work. We would also like to thank Allison Miller for comments on an early draft of the paper. T.\ M.\ was supported in part by NSF grant DMS-1309212; L.\ P.\ was partially supported by an NSF GRFP fellowship; F.\ V.\ was partially supported by an AMS--Simons Travel Grant.  
\section{Proof of Theorem \ref{Thm:largegaps}}\label{Sec:largegaps}
Consider the knots $K_n$ and $K_{n,t}^s$ as in Figure~\ref{fig:KnKnt} for $n,t,s \in \bZ$, where labeled boxes represent positive half twists, as illustrated in Figure~\ref{HalfTwist}, and t is even.
\begin{figure}[t]
    \centering
   \includegraphics[scale=.4]{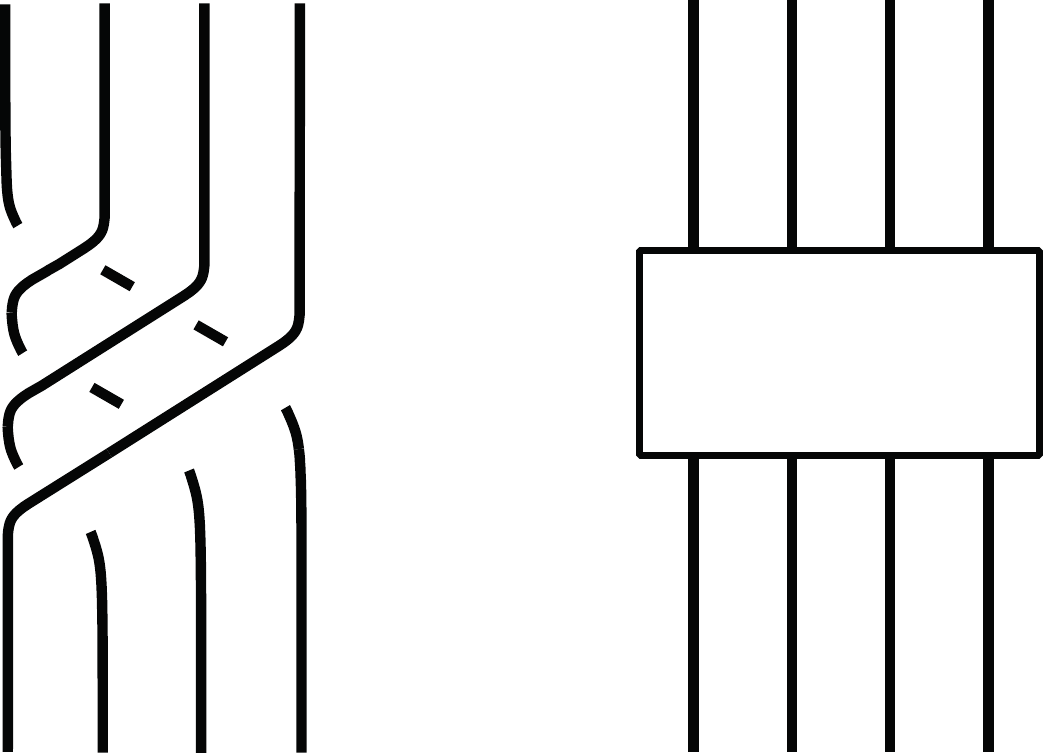}
\put(-27,42){\Large$1$}
\put(-75,42){\Huge{=}}
    \caption{In all the knot diagrams throughout the paper, labeled boxes represent positive half twists. Here, an example is depicted.}
    \label{HalfTwist}
\end{figure}
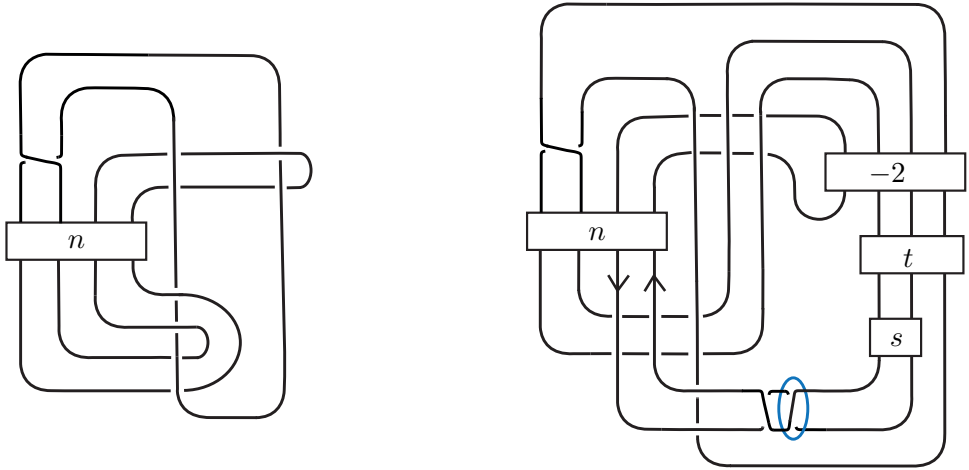
\begin{figure}[b]
    \centering
   \includegraphics[scale=.55]{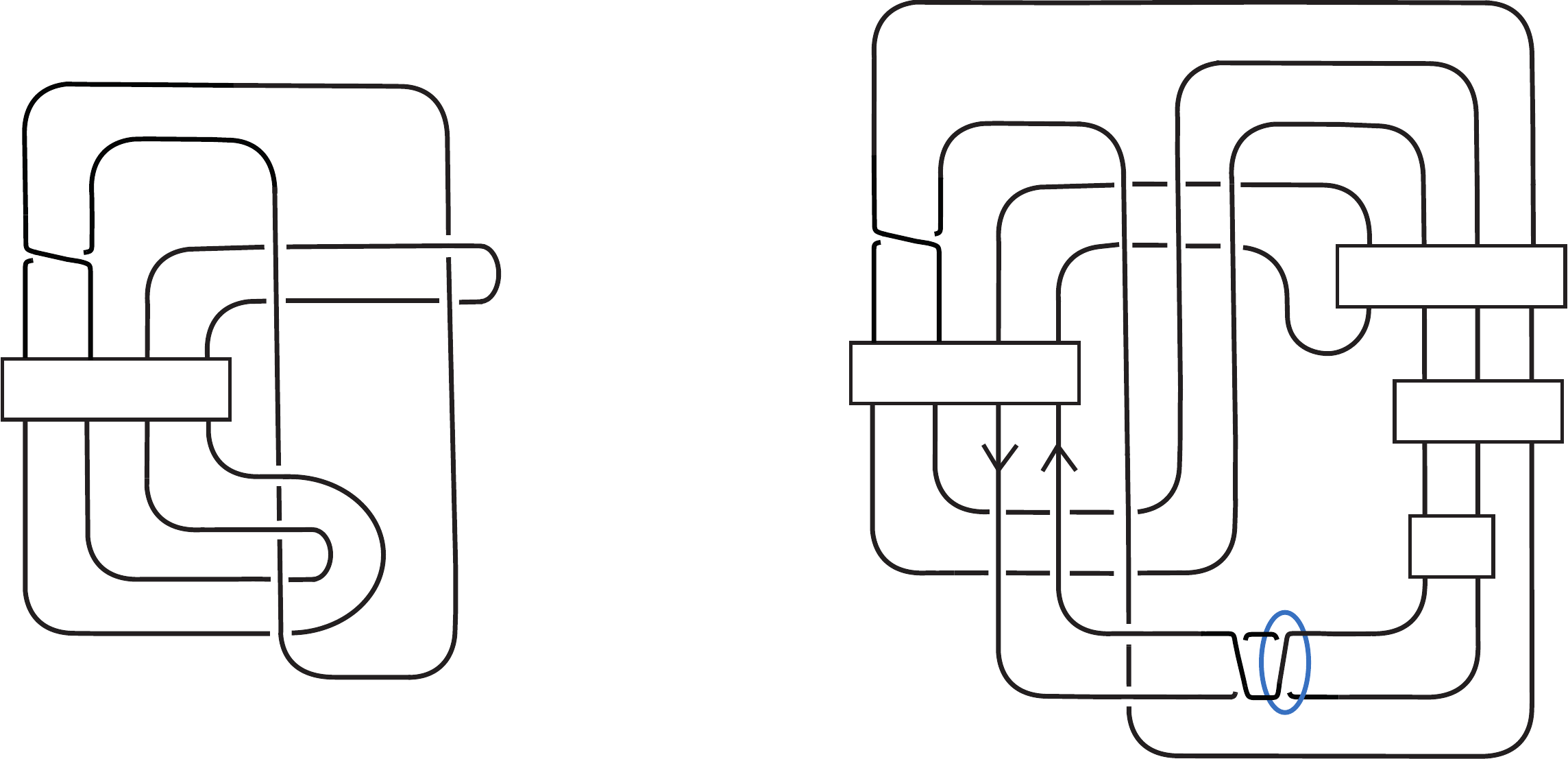}
\put(-340,83){$n$}
\put(-143,86){$n$}
\put(-37,108){$-2$}
\put(-24,76){$t$}
\put(-29,46){$s$}
    \caption{The knots $K_n$ (left) and $K_{n,t}^s$ (right).}
    \label{fig:KnKnt}
\end{figure}

The essential point in the proof is that $X_{t/2}(K_n)$ is diffeomorphic to $X_{t/2}(K_{n,t}^0)$. To see this, observe that in the terminology of \cite{AJLO14} the diagram on the left of Figure \ref{fig:KnKnt} is a simple annulus presentation for $K_n$. Further, the knot $K^0_{n,t}$ is obtained from $K_n$ by ($*\frac{t}{2}$)-twisting, which is defined by \cite{AJLO14} and is a natural modification of the annulus twisting defined by \cite{Oso06}. Therefore Theorem 3.10 of \cite{AJLO14} implies that $X_{t/2}(K_n) \cong X_{t/2}(K_{n,t}^0)$.

In the notation of the introduction, we take $P_m = K_{m-3}$ and $Q_m = K_{m-3, -2m}^0$ so we have $X_{-m}(P_m) \cong X_{-m}(Q_m)$. The main work toward Theorem \ref{Thm:largegaps} is contained in the following estimates on the Thurston-Bennequin numbers of $K_n$ and $K_{n,t}^s$.

\begin{lem}
For $n\ge 0$, $\overline{tb}(K_n)= -2-n$.
\label{Lem:tbKn}
\end{lem}


\begin{thm}
For $n\ge 0$, $\overline{tb}(K_{n,2(-3-n)}^0)\le -3-2n$.
\label{Thm:tbKnts}
\end{thm}

With these results and the preceding remarks, the proof of Theorem \ref{Thm:largegaps} is done:

\begin{proof} [Proof of Theorem \ref{Thm:largegaps}]
For $m\geq 3$, let $Q_m = K_{m-3,-2m}^0$, and $P_m = K_{m-3}$. Then by the previous results $\maxtb(P_m) = -m+1$ and $\maxtb(Q_m) \leq -2m+3$.  Since $-m<\maxtb(P_m)$, the common 2-handlebody $X_{-m}(P_m)\cong X_{-m}(Q_m)$ is Stein. Therefore $\sf(Q_m) \geq -m$, hence
\[
\sf(Q_m) - \maxtb(Q_m) \geq -m - (-2m+3) = m-3.
\]
Take $J_m = Q_{m+3}$.
\end{proof}
Theorem \ref{Thm:tbKnts} is the main technical result; Subsection \ref{Subsec:tbbound} is dedicated to its proof. 
\subsection{Input from Khovanov homology}\label{Subsec:background}

In this subsection we briefly recall the background we need to prove Theorem~\ref{Thm:tbKnts}. We mainly use the notation of~\cite{Ng05}. 

Khovanov homology is an invariant of oriented links in $S^3$ which associates to a link $L$ a bigraded abelian group $HKh^{i,j}(L)$~\cite{Kho00}. We will be concerned in particular with Khovanov homology collapsed to a single grading $v=i-j$, which we will denote $HKh^{*}(L)$. It will be convenient to take the tensor product $HKh^{*}(L)\otimes \bQ$. We still denote the tensor product by $HKh^{*}(L)$. 

Recall that an oriented Legendrian link $\mathcal{L}$ in $S^3$ equipped with the standard contact structure admits a front projection to the $xz$ plane with singularities consisting of only double points and cusps, and without vertical tangencies~\cite{atkowski}. The Legendrian condition means that at a double point the strand with the lower slope passes in front of the other strand, recovering crossing information at the double points and yielding an oriented link diagram with cusps and no vertical tangencies, denoted $F$. Let $w(F)$ be the writhe of $F$ (the signed number of crossings) and let $c(F)$ be half the number of cusps of $F$. If $\L$ has a single component, then the Thurston-Bennequin number of $\mathcal{L}$ agrees with $w(F) - c(F)$. See, for instance,~\cite{Etn03}.

In~\cite{Ng05}, Ng gives an upper bound for $\overline{tb}(K)$ in terms of data provided by the Khovanov homology of $K\subset S^3$. 

\begin{defin} For a link $L$ in $S^3$, define 
\[
\kappa(L):=\mathrm{min}\{ v | HKh^v(L)\neq 0\}.\]
\end{defin}

\begin{thm}[Corollary 2 of \cite{Ng05}] For a knot $K\subset S^3$, 
$$\overline{tb}(K)\le \kappa(K).$$
\label{Thm:Ngbound}
\end{thm}
One method for calculating $\kappa(L)$ is to resolve  the link $L$ into simpler links and use a long exact sequence. Figure~\ref{resolutions} depicts two resolutions of a crossing $c$ of a diagram $D$ of $L$; we denote the resolutions by $Res_0(D, c)$ and $Res_1(D, c)$. We drop $(D,c)$ from the notation when the diagram $D$ and the specified crossing $c$ are understood from context. 
\begin{figure}[t]
\includegraphics[scale=.55]{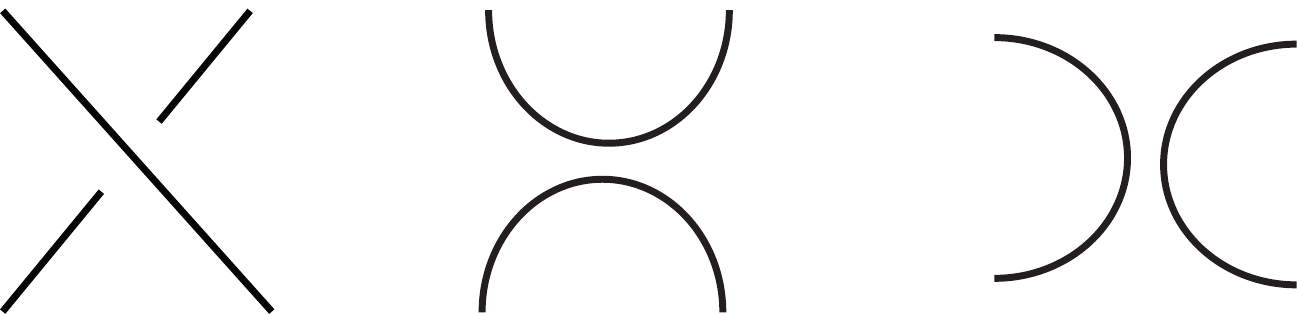}
\caption{The crossing $c$ in a diagram $D$ with 0 and 1 resolutions, respectively.}
\label{resolutions}
\end{figure}

\begin{lem}[Lemma~6 of~\cite{Ng05}]\label{Lem:LES}
There is a long exact sequence of the form
\begin{equation}\label{LES}
\xymatrix{ \KH(Res_0(D)) \ar[rr]^{(r)} &&
\KH(Res_1(D))
\ar[dl] ^{(p)}\\
& \KH(D) , \ar[ul]^{(q)} & }
\end{equation}
with grading shifts $q=w(Res_0(D))-w(D), p=w(D)-w(Res_1(D))$, and $r=-1-p-q$. 
\end{lem}


{\rmk \label{rmk:orientations} These resolutions will not necessarily induce a well-defined orientation on $Res_0$ and $Res_1$. On those components of $Res_0$ (respectively $Res_1$) that do not inherit a well-defined orientation, we may choose any arbitrary orientation. Although the shifts $p, q, r$ in the sequence depend on the chosen orientations, the gradings on HKh of these resolutions also depend on these orientations in a corresponding way. As such, the conclusions one draws about $HKh^*(D)$ from the sequence are independent of the choice of orientation on the resolutions.}

We will prove Theorem~\ref{Thm:tbKnts} by computing $\kappa(K_{n,2(-3-n)}^0)$ and applying Theorem~\ref{Thm:Ngbound}. The calculation proceeds inductively using Lemma \ref{Lem:LES}, but we will also employ Dror Bar-Natan's Fast KH routines (available at \cite{KAT}) to compute $\kappa$ for some small examples. We will say ``via computer'' throughout the paper to indicate when $\kappa$ was computed with these routines.  

{\rmk \label{rmk:Hopf}In using Lemma~\ref{Lem:LES}, we often get a Hopf link after resolving a crossing in our examples. Define $H^+$ and $H^-$ to be the oriented Hopf links as in Figure~\ref{fig:hopfs}. It is easy to compute directly, or see \cite{BN02} or \cite{KAT}, that $\kappa (H^+)=0$ and $\kappa(H^-)=-4$.}
\begin{figure}[h]
    \centering
   \includegraphics[scale=.4]{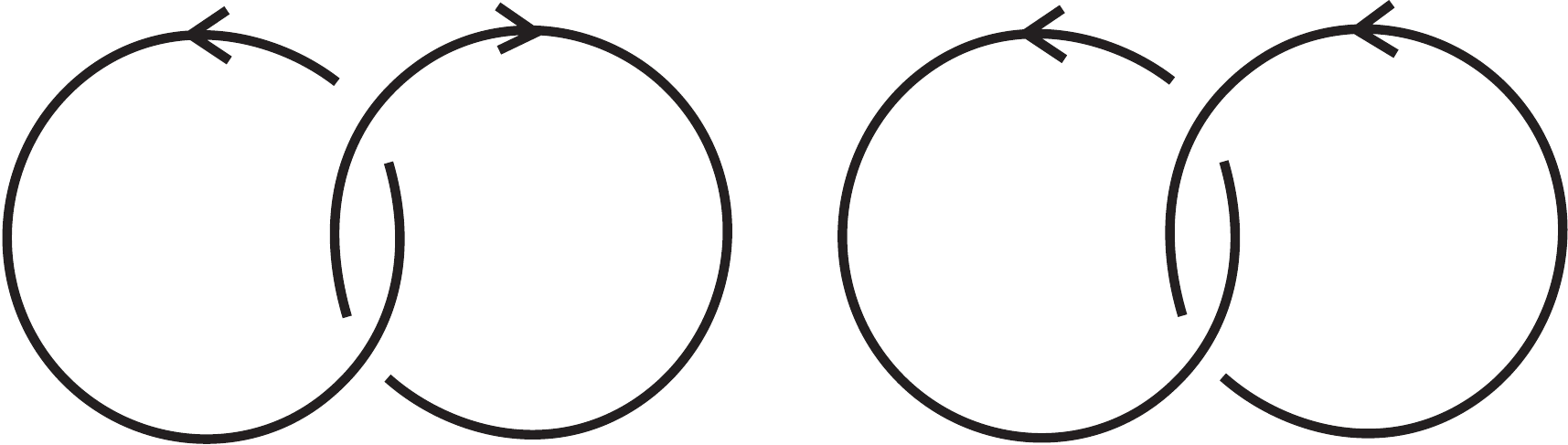}
    \caption{$H^+$ and $H^-$ respectively.}
    \label{fig:hopfs}
\end{figure}
\subsection{Computing a Thurston-Bennequin upper bound}\label{Subsec:tbbound}

This section is devoted to proving Theorem~\ref{Thm:tbKnts} by computing $\kappa(K_{n,t}^0)$. In fact, Theorem \ref{Thm:tbKnts} follows immediately by taking $t = 2(-3-n)$ in the following theorem, and applying Theorem \ref{Thm:Ngbound}.

\begin{thm}\label{Thm:kappaKnts} For any $n\geq 0$ and any even $t\leq 2$, we have $\kappa(K^0_{n,t}) = -3-2n$.
\end{thm}

The overall structure of the proof of Theorem \ref{Thm:kappaKnts} is a decreasing induction on $t$ for even $t\leq 2$. We will also need to see that $\kappa(K^0_{n,4}) = -2-n$; both this and the base case $t = 2$ of Theorem \ref{Thm:kappaKnts} are proved by induction on $n$. At several points in the proof the knots $K_{n,t}^s$ with $s = 4$ will arise, so we begin with the following lemma relating the values of $\kappa(K_{n,t}^4)$ and $\kappa(K_{n,t}^0)$. The structure of the argument serves as a model for many similar proofs to follow.

 \begin{lem}
If $\kappa(K_{n,t}^0)\le -3$, then $\kappa(K_{n,t}^4)= \kappa(K_{n,t}^0) - 4$. If $\kappa(K_{n,t}^{4})\le -3$, then $\kappa(K_{n,t}^0)= \kappa(K_{n,t}^{4}) +4$. 
\label{Lem:changes}
\end{lem}
\begin{proof}

We will prove the first assertion of the lemma explicitly. The second follows similarly and we leave the proof to the reader. We claim that when $s$ is even there is a long exact sequence of the form
\begin{equation}\label{eq:1LES}
\xymatrix{\KH(H^+) \ar[rr]^{(-5)} &&
\KH(K_{n,t}^{s+1})
\ar[dl] ^{(5)}\\
& \KH(K_{n,t}^s), \ar[ul]^{(-1)} & }
\end{equation}
and when $s$ is odd there is a long exact sequence of the form
\begin{equation}\label{eq:2LES}
\xymatrix{\KH( H^-) \ar[rr]^{(3)} &&
\KH(K_{n,t}^{s+1})
\ar[dl] ^{(-3)}\\
& \KH(K_{n,t}^s). \ar[ul]^{(-1)} & }
\end{equation}

Indeed, these are the long exact sequences associated to the crossing $c$ indicated in the diagram for $K_{n,t}^s$ of Figure \ref{fig:KnKnt}; we merely must check details. Assume $s$ is even: then by comparing signs of crossings in the diagram before and after resolving $c$, we find that in notation of Lemma \ref{Lem:LES}, we have $p=5$ and $q=-1$. It is also easy to see by inspecting the diagram that $Res_0 \simeq H^+$ and $Res_1\simeq K_{n,t}^{s+1}$, where $\simeq$ denotes isotopy. The sequence \eqref{eq:1LES} follows; sequence \eqref{eq:2LES} is similar.

Now the proof proceeds by appealing to the sequences above four times. Take $s = 0$ in~\eqref{eq:1LES}, and recall that from Remark~\ref{rmk:Hopf} we have $HKh^u(H^+) = 0$ for all $u\leq -1$. Hence for $u\le -1$, 
\[
HKh^{u-5}(K_{n,t}^1)\cong HKh^{u}(K_{n,t}^0).
\]
Using the hypothesis that $\kappa(K_{n,t}^0)\le -3$, this implies
\begin{equation}\label{eq:kappa}
 \kappa (K_{n,t}^1)=\kappa (K_{n,t}^0)-5.
\end{equation}
Now using~\eqref{eq:2LES}~and~Remark~\ref{rmk:Hopf}, we get that for $u\le -5$
\[
HKh^{u+3}(K_{n,t}^{2})\cong HKh^{u}(K_{n,t}^{1}).
\]
Hence $$\kappa(K_{n,t}^2)= \kappa(K_{n,t}^1)+ 3,$$ since the hypothesis implies $\kappa(K_{n,t}^1)\le -8$ (in fact, we need only $\kappa(K_{n,t}^1)\leq -6$). Combining this with ~\eqref{eq:kappa} we see that
\[
\kappa(K_{n,t}^2)=\kappa(K_{n,t}^0)-2.
\]
Note that in particular $\kappa(K_{n,t}^2)\le -3$, so the same argument can be repeated starting with $s = 2$, and the result follows. 
\end{proof}

Now we begin our inductive calculation of $\kappa(K_{n,t}^0)$. The first, special, case is $t = 4$, and here there is a simplification: we have that $K_{n,4}^0$ is isotopic to the knot $K_n$ on the left of Figure \ref{fig:KnKnt}. The isotopy is indicated in Figure \ref{Isotopy}. 

\begin{lem}\label{Lem:kappaKn} $\kappa(K_n) = -2-n$.
\end{lem}

\begin{proof} We proceed by induction on $n$. The base case $n=0$ follows by the observation that $K_0$ is the mirror of the knot $8_{20}$ in Rolfsen's table, whose maximum Thurston-Bennequin number and also $\kappa$-invariant are equal to $-2$ (see \cite{Ng05} and \cite{Ng01}, or one can check by computer).
\begin{figure}[t]
    \centering
   \includegraphics[scale=.55]{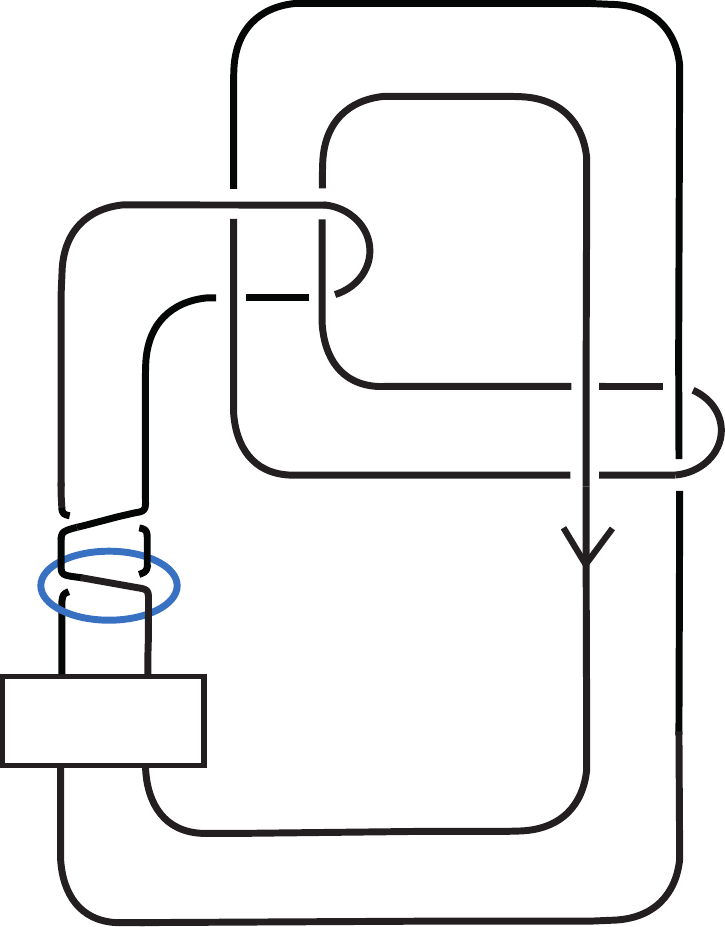}
\put(-110,29){\small{$n-1$}}
    \caption{A diagram of the knot $K_{n-1}$.}
    \label{fig:Kn-1}
\end{figure}
Observe that a sequence of isotopies (specifically, flypes) brings the diagram of $K_{n-1}$ to the one in Figure~\ref{fig:Kn-1}, and let $c$ be the indicated crossing of this diagram of $K_{n-1}$. In the associated long exact sequence~\eqref{LES}, we find $p=1$, $q=-1$ and $Res_1\simeq K_n$. 
The 0 resolution $Res_0$ is isotopic to a twisted Whitehead link, independent of $n$: a bit of additional work with \eqref{LES} or a computer calculation shows that $\kappa(Res_0)=0$. Therefore, when $u\le -1$ we have that, $$HKh^{u-1}(K_{n})\cong HKh^u(K_{n-1}).$$ 
Using the induction hypothesis we infer $\kappa(K_n) = \kappa(K_{n-1}) - 1$, from which the lemma follows.
\end{proof}

\begin{figure}[t]
    \centering
   \includegraphics[scale=1]{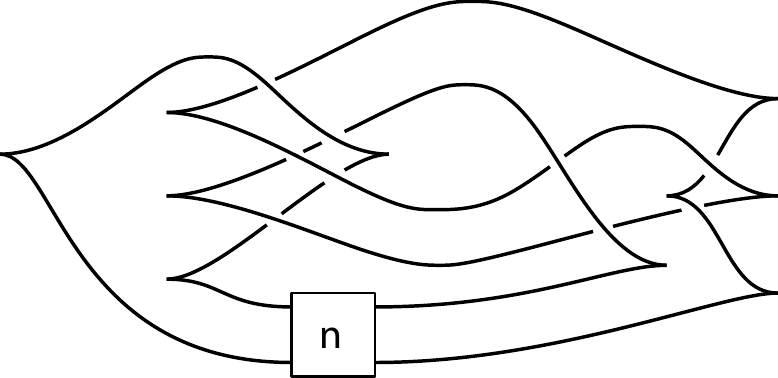}
    \caption{A Legendrian diagram of $K_n$. Note that $tb=-2-n$.}
    \label{fig:Legendrian}
\end{figure}

We pause to observe that the preceding lemma easily gives Lemma \ref{Lem:tbKn} calculating the maximum Thurston-Bennequin number of $K_n$:

\begin{proof}[Proof of Lemma \ref{Lem:tbKn}] The diagram of Figure \ref{fig:Kn-1} (replacing $n-1$ by $n$) can easily be turned into the Legendrian diagram of Figure~\ref{fig:Legendrian}. Since the Legendrian in that diagram has Thurston-Bennequin number $-2-n$,  the result  follows from Lemma \ref{Lem:kappaKn} and Theorem~\ref{Thm:Ngbound}. 
\end{proof}

We now turn our attention to calculating $\kappa(K_{n,2}^0)$, for which the argument is a bit more involved. First we leave it as an exercise for the reader to check that Figure \ref{fig:Kn2} is a diagram for $K_{n,2}^0$ (the reader who would like a hint can consider the related isotopy in Figure \ref{Isotopy2}). Our argument proceeds by applying the long exact sequence \eqref{LES} to the crossing indicated in Figure~\ref{fig:Kn2}.  
\begin{figure}[b]
    \centering
   \includegraphics[scale=.55]{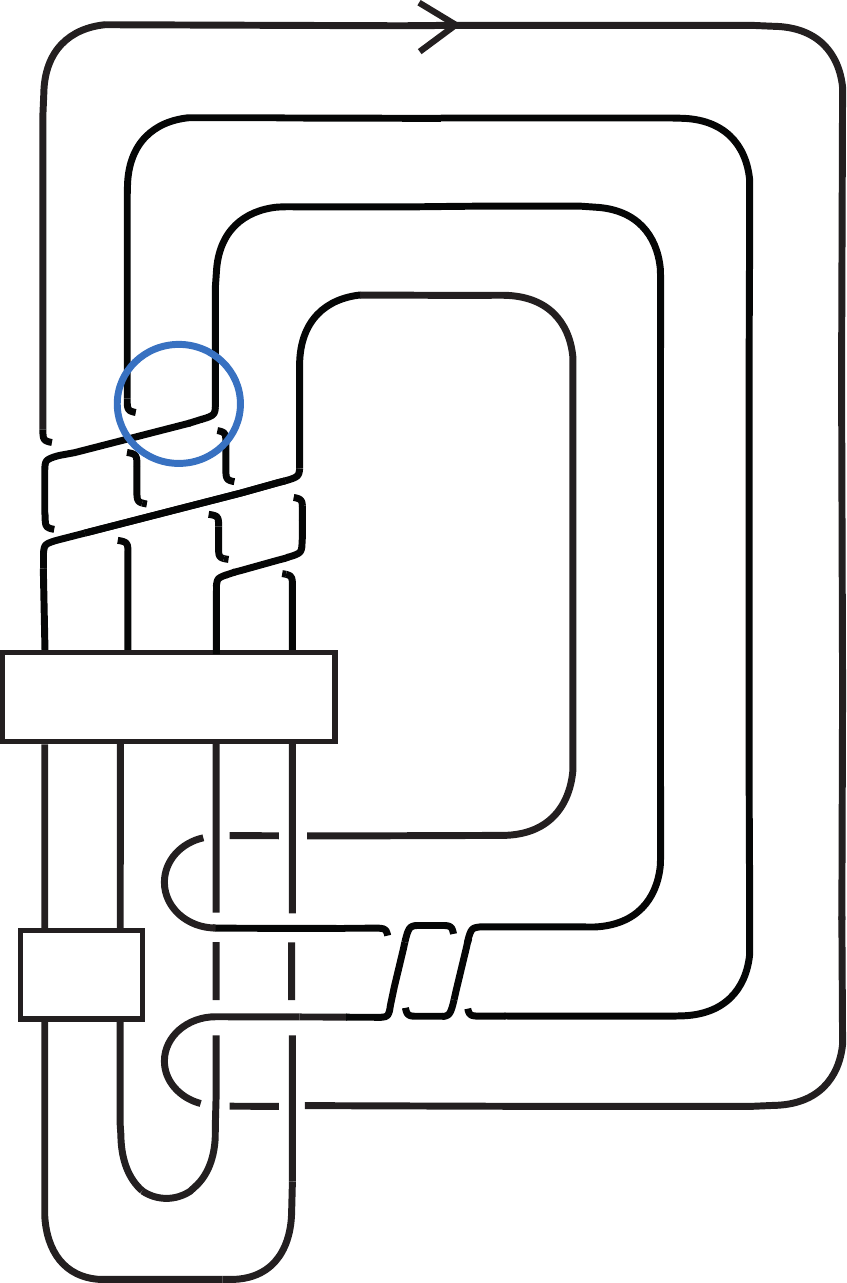}
\put(-120,90){$n+1$}
\put(-130,46){$-4$}
    \caption{A diagram of the knot $K_{n,2}^0$.}
    \label{fig:Kn2}
\end{figure}
Now, the 1 resolution of that crossing yields a link isotopic to the negative torus link $T(-4,2)$, independent of $n$. On the other hand, $Res_0\simeq R_n$, where $R_n$ is depicted in Figure~\ref{fig:Dn0}. 
\begin{figure}[t]
    \centering
   \includegraphics[scale=.55]{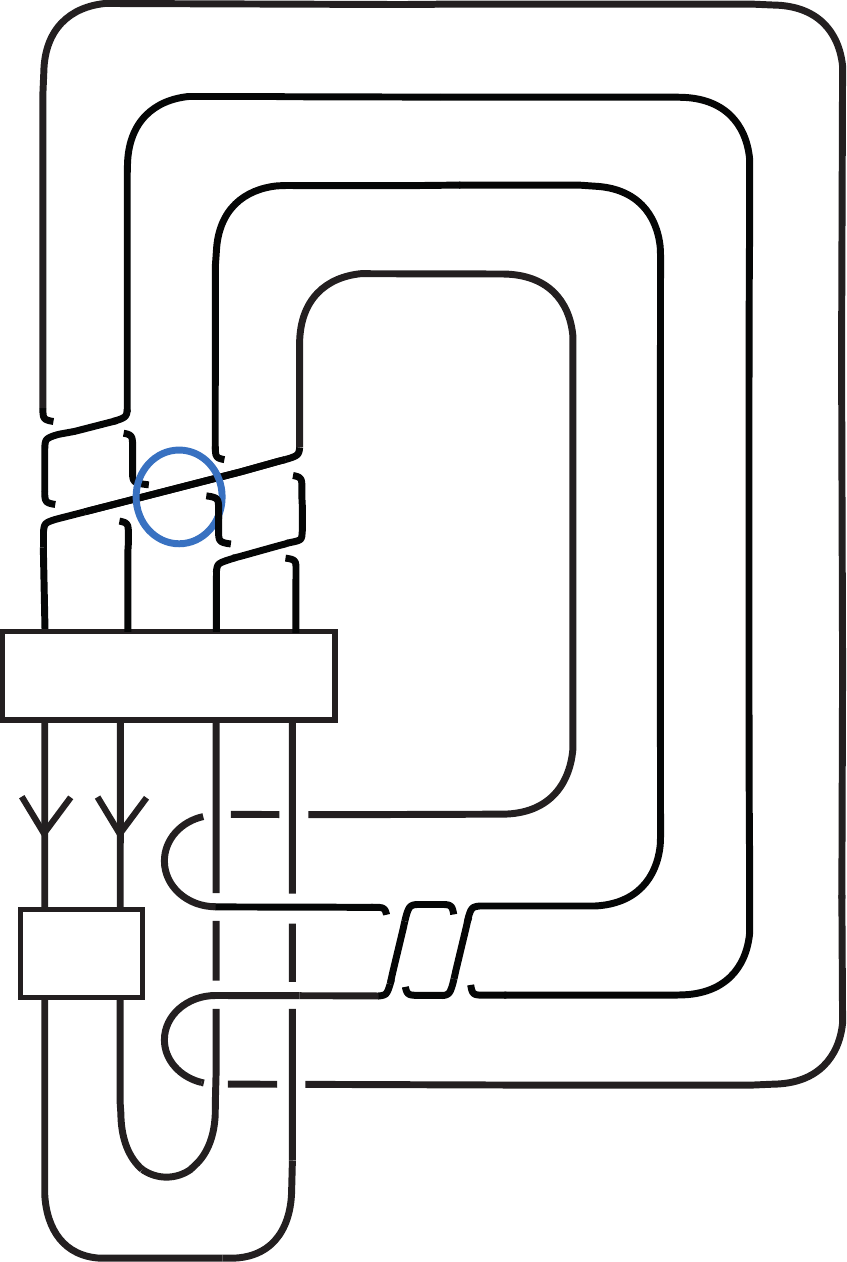}
\put(-120,90){$n+1$}
\put(-129,45){$-4$}
    \caption{A diagram of the knot $R_n$.}
    \label{fig:Dn0}
\end{figure}

\begin{figure}[t]
    \centering
   \includegraphics[scale=.55]{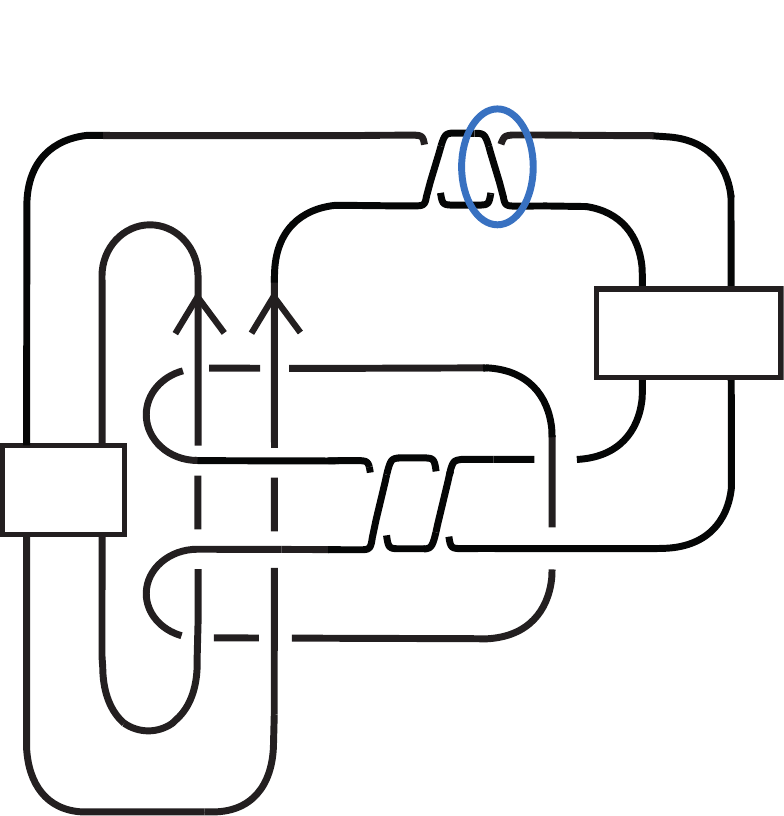}
\put(-122,49){\small{$-2$}}
\put(-27,73){\small{$n+1$}}
    \caption{A diagram of the link $Q_n$.}
    \label{fig:Qn}
\end{figure}
Taking resolutions of the crossing indicated in Figure \ref{fig:Dn0}, we find that the 1 resolution gives a link $Q_n$ as in Figure \ref{fig:Qn}, while the 0 resolution is isotopic to $K^4_{n-1,2}$ (this isotopy is indicated in Figure \ref{Isotopy2}). Our strategy for calculating $\kappa(K^0_{n,2})$ is summarized in the resolution tree shown in Figure \ref{fig:Resolutions3}.
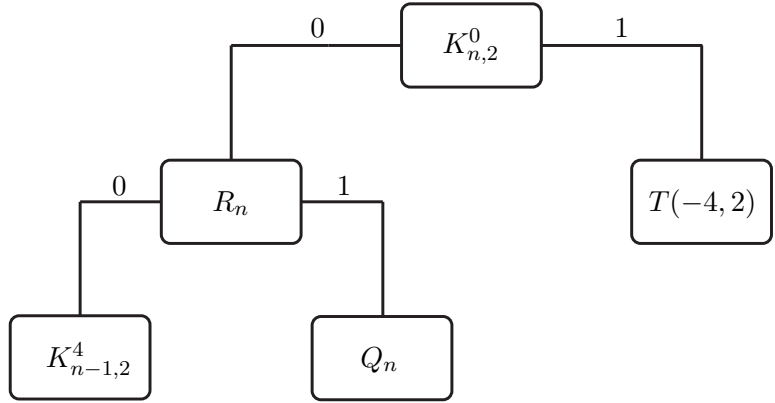
\begin{figure}[tb]
    \centering
   \includegraphics[width=4in]{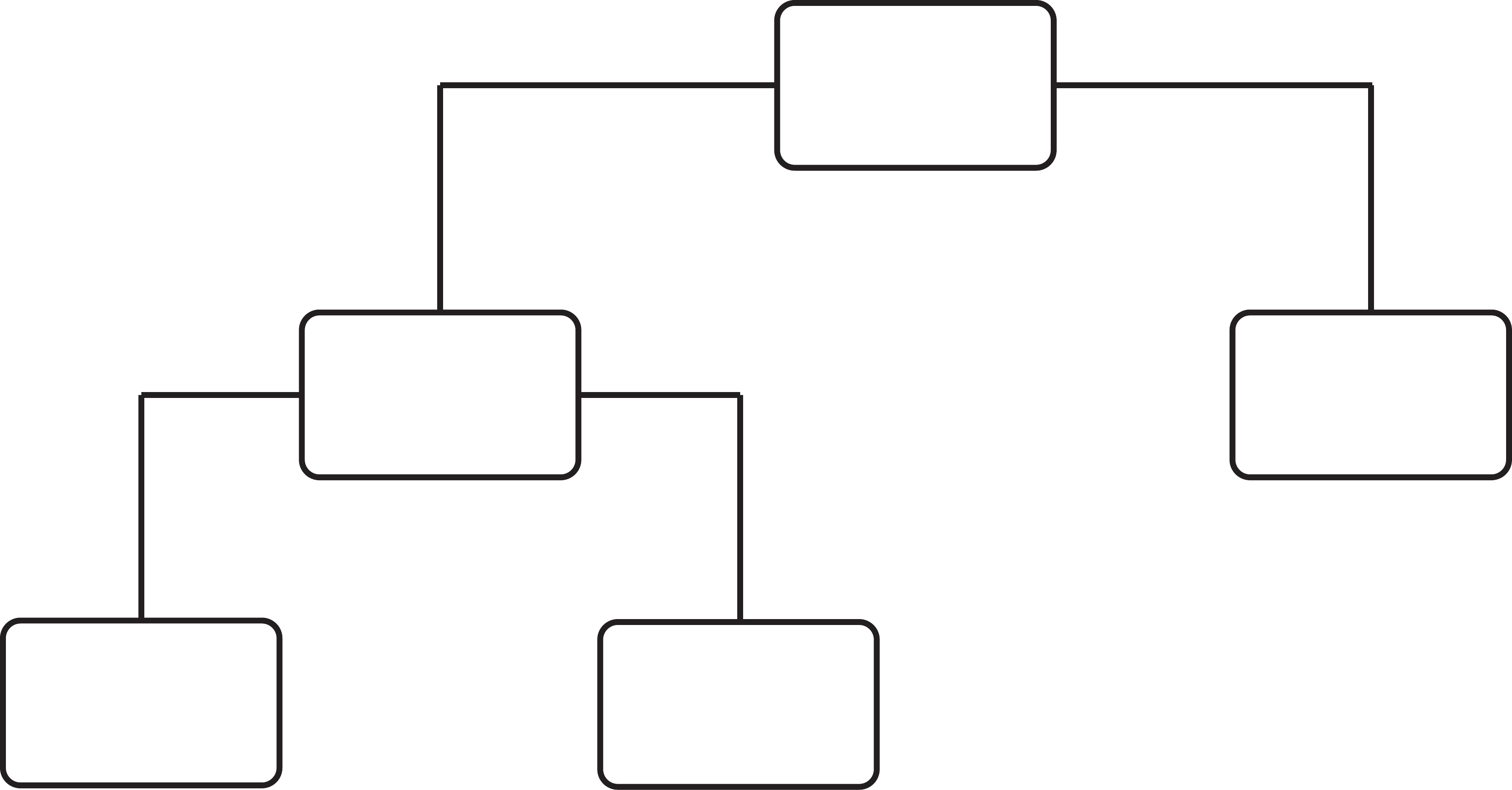}
\put(-125,132){$K_{n,2}^0$}
\put(-175,138){$0$}
\put(-60,138){$1$}
\put(-212,72){$R_n$}
\put(-47,72){$T(-4,2)$}
\put(-275,13){$K^4_{n-1,2}$}
\put(-156,13){$Q_n$}
\put(-250,78){$0$}
\put(-165,78){$1$}
    \caption{Resolution tree yielding the calculation of $\kappa(K^0_{n,2})$.}
    \label{fig:Resolutions3}
\end{figure}

We begin at the bottom of the resolution tree.

\begin{lem}
$\kappa(Q_n)=-8-n$. 
\label{Lem:kappaQn}
\end{lem}

\begin{proof}
We will prove this by induction on $n$, and check the case $n=0$ via computer. 

Apply the long exact sequence of \eqref{LES} to the crossing indicated in Figure~\ref{fig:Qn}: we find that $p=-1$, $q=1$ and $Res_0 \simeq Q_{n-1}$. The 1 resolution gives a 3-component link independent of $n$, and we check via computer that $\kappa(Res_1)=-7$. See Remark~\ref{rmk:orientations}. 

Hence when $r\le-8$ we get that $$HKh^{r-1}(Q_{n})\cong HKh^r(Q_{n-1}),$$ which, with the induction hypothesis, gives us that $$\kappa(Q_n)=\kappa(Q_{n-1})-1=-8-n.$$ 
\end{proof}

The other terms in the resolution tree are either straightforward (for $T(-4,2)$) or rely on the induction hypothesis (for $K^4_{n-1,2}$, and therefore $R_n$ as well). We therefore give the remainder of the proof all at once.

\begin{lem}
$\kappa(K_{n,2}^0)=-2n-3$. 
\label{Lem:kappaKn20}
\end{lem}
\begin{proof}
As before we check the base case $n=0$ via computer. 
In the long exact sequence associated to the crossing indicated in Figure~\ref{fig:Kn2}, we find that the writhe differences are $p=-1$ and $q=-7$.  For the torus link $T(-4,2)$ it is not hard to check directly (or one can verify by computer) that $\kappa = 0$, and therefore the long exact sequence shows 
 \[
HKh^u(K_{n,2}^0)\cong HKh^{u-7}(R_n)
\] 
for $u\le -2$.  In particular, so long as $\kappa(R_n)\leq -9$, we have 
\begin{equation}\label{eq:kappa1}
\kappa(K^0_{n,2}) = \kappa(R_n) + 7.
\end{equation}
Now consider the long exact sequence arising from the crossing of $R_n$ indicated in Figure \ref{fig:Dn0}. This time we have $p=-1$ and $q=5$. From the fact that $\kappa(Q_n) = -8-n$, we know that for $u<-8-n$, 
\[
HKh^{u-1}(R_n)\cong HKh^{u+4}(K_{n-1,2}^4). 
\]
It follows that so long as $\kappa(K^4_{n-1,2})\leq -5-n$ we get 
\begin{equation}\label{eq:kappa2}
\kappa(R_n) = \kappa(K^4_{n,2}) -5.
\end{equation}

We recall from Lemma \ref{Lem:changes} that if $\kappa(K_{n,2}^0)\le-3$ then $\kappa(K_{n,2}^4)=\kappa(K_{n,2}^0)-4$, and with this we can give the induction.

Suppose that the lemma is proved for $K^0_{n-1,2}$ for some $n\geq 1$, so that $\kappa(K^0_{n-1,2}) = -2n-1\leq -3$. Then from Lemma \ref{Lem:changes} we have $\kappa(K^4_{n-1,2}) = -2n-5$, and in particular this says $\kappa(K^4_{n-1,2}) \leq -5-n$. Hence from \eqref{eq:kappa2} we have $\kappa(R_n) = -2n-10$.

Then in particular we have $\kappa(R_n)\leq -9$, so from \eqref{eq:kappa1} we get $\kappa(K^0_{n,2}) =  \kappa(R_n) + 7 = -2n-3$ as desired.
\end{proof}

\begin{figure}[h]
    \centering
   \includegraphics[scale=.52]{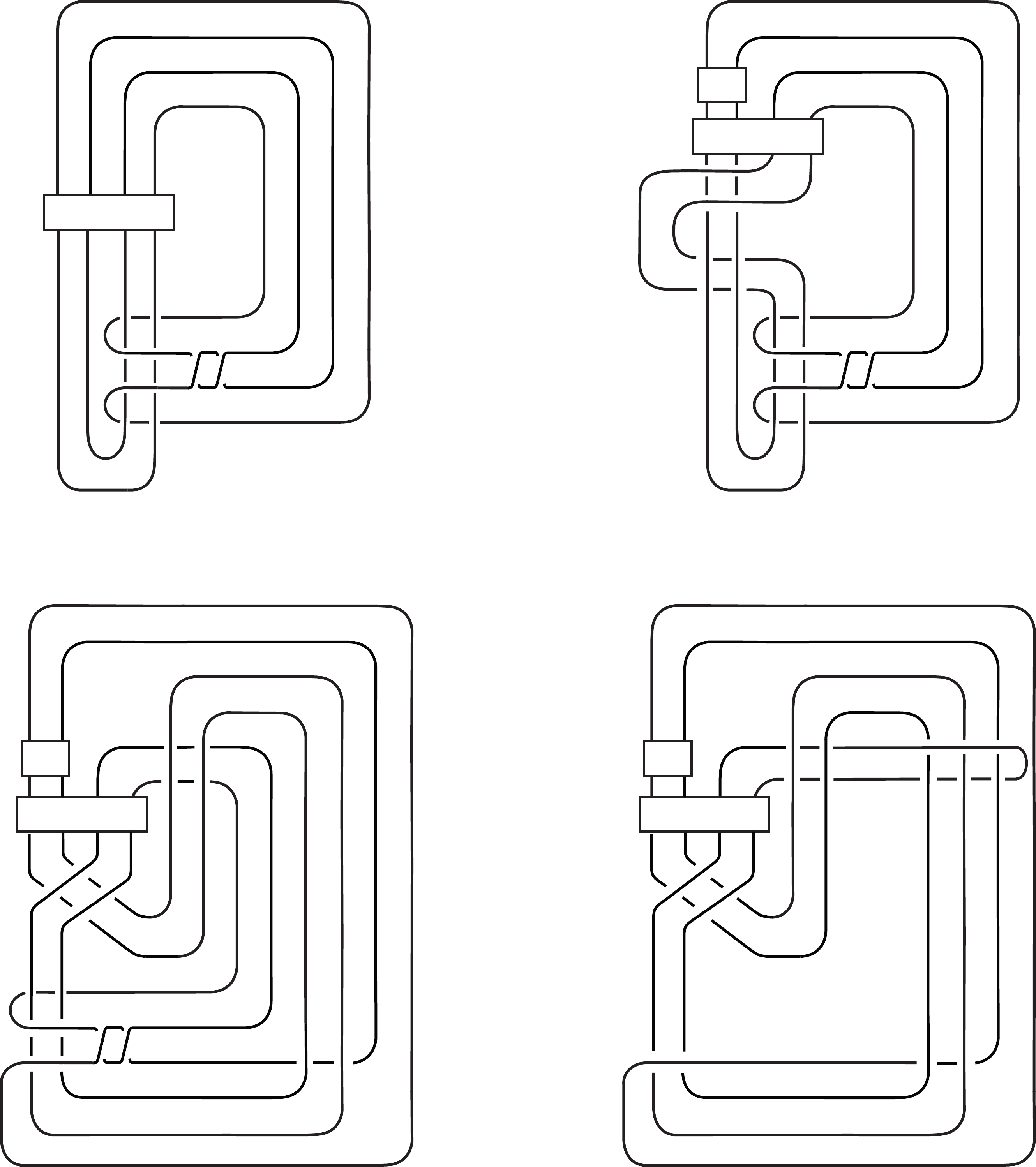}
\put(-400,245){(a)}
\put(-170,245){(b)}
\put(-400,-20){(c)}
\put(-170,-20){(d)}
\put(-370,365){$n+1$}
\put(-120,394){$n-1$}
\put(-125,414){$4$}
\put(-380,132){$n-1$}
\put(-385,153.5){$2$}
\put(-140,132){$n-1$}
\put(-145,153.5){$2$}
    \caption{The series of isotopies needed to go from the 0 resolution of $R_n$ to $K_{n-1,2}^4$. We leave it to the reader to check that the diagram obtained from the 0 resolution of the marked crossing in Figure~\ref{fig:Dn0} is isotopic to the first picture (top left), also that the last picture (bottom right) is isotopic to $K_{n-1,2}^4$.}
    \label{Isotopy2}
\end{figure}
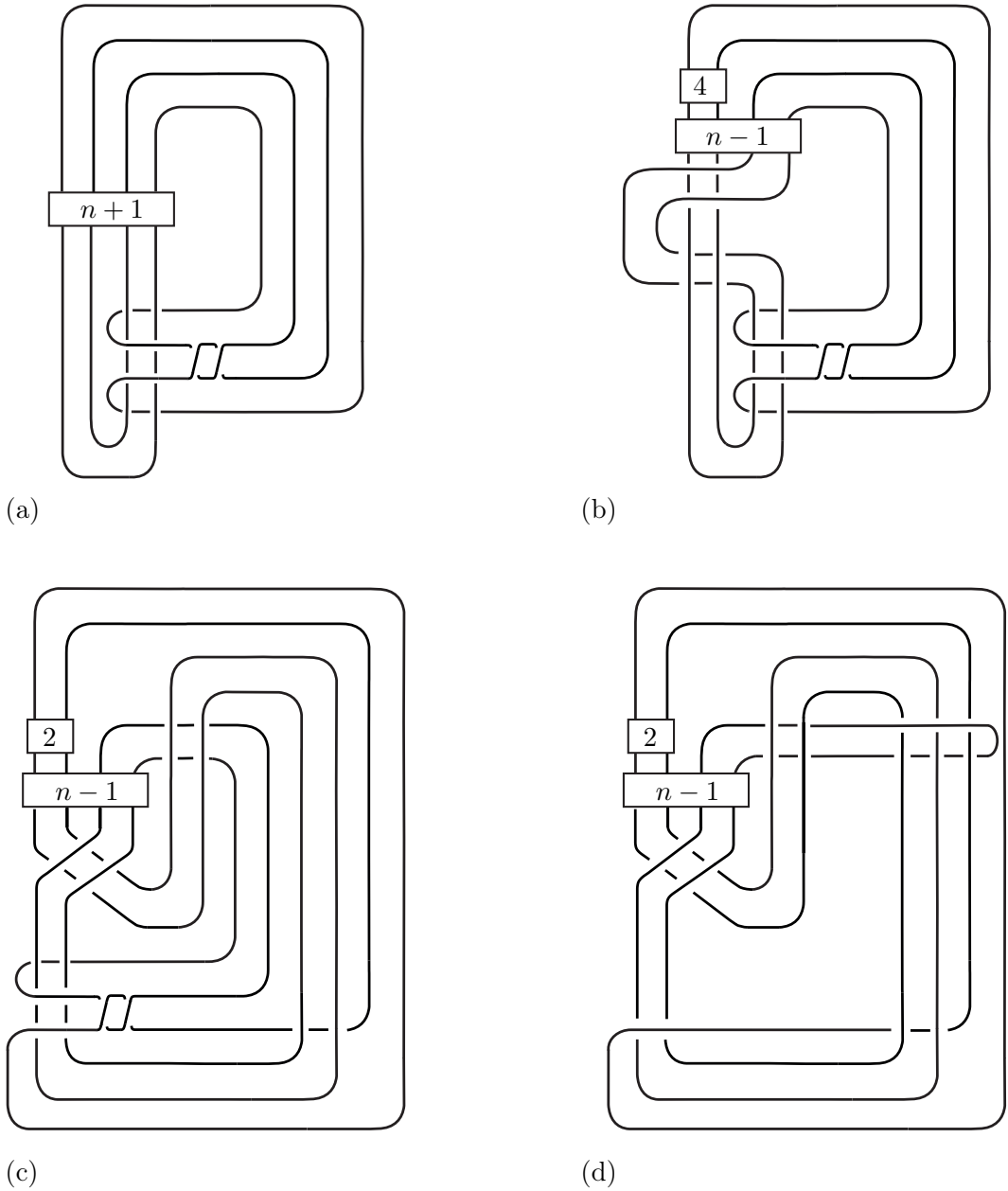

We can now give the proof of Theorem \ref{Thm:kappaKnts}. We use the diagram for $K_{n,t}^0$ in Figure \ref{fig:Knt}, and note that the 1 resolution of the crossing specified in that figure changes $K_{n,t}^0$ into a knot isotopic to $K_{n,t-4}^4$. The  0 resolution gives us the link in Figure~\ref{fig:KntH}, denoted $K_{n,t}^H$. Resolving the crossing specified in Figure~\ref{fig:KntH} will either give $K_{n,t-2}^0$ or $K_{n,t-2}^4$. Figure~\ref{fig:Resolutions1} illustrates this resolution tree.  
\begin{figure}[tb]
    \centering
   \includegraphics[width=4in]{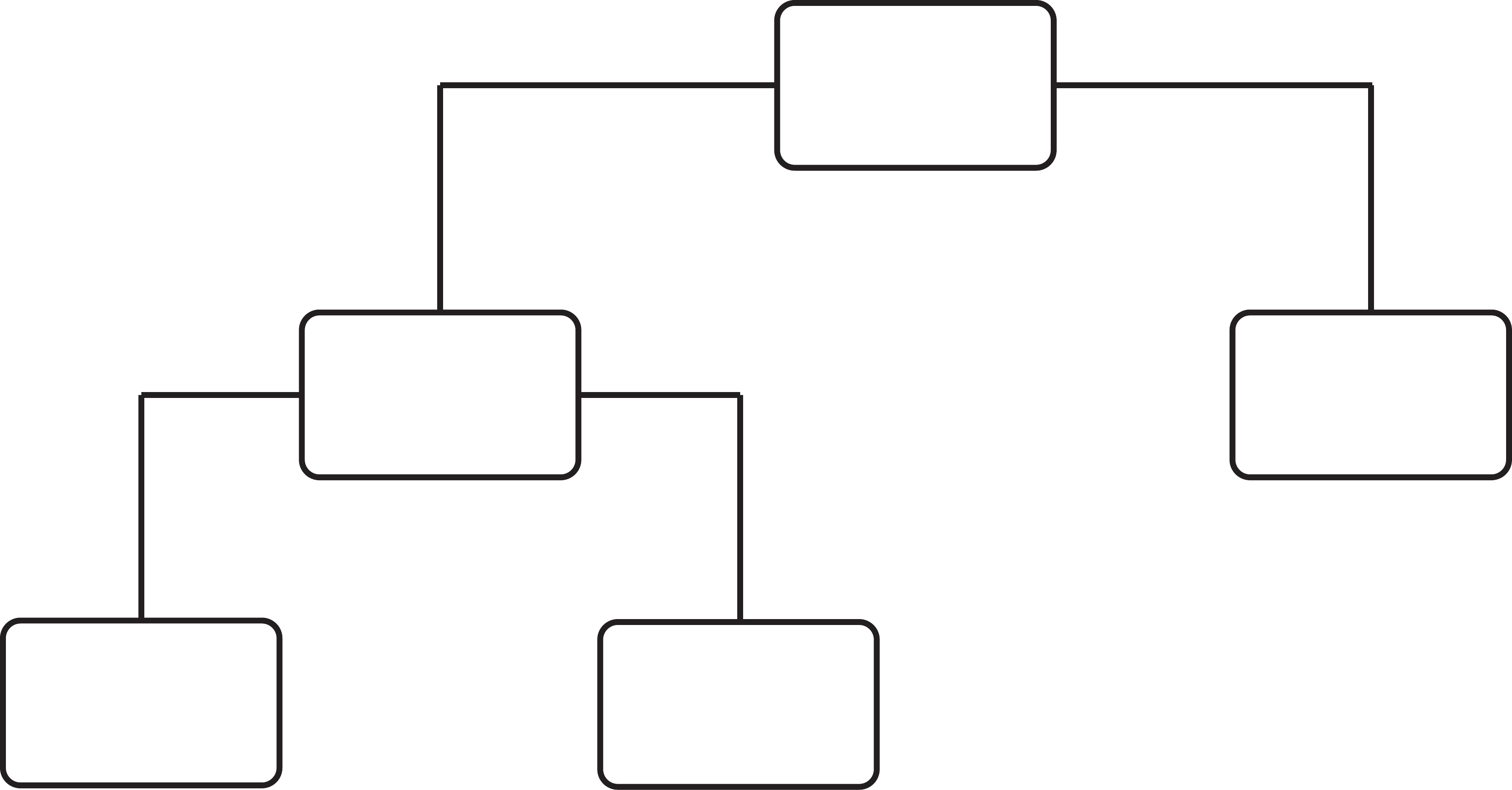}
\put(-125,132){$K_{n,t}^0$}
\put(-175,138){$0$}
\put(-60,138){$1$}
\put(-215,70){$K_{n,t}^H$}
\put(-40,70){$K_{n,t-4}^4$}
\put(-275,13){$K_{n,t-2}^4$}
\put(-160,13){$K_{n,t-2}^0$}
\put(-250,78){$0$}
\put(-165,78){$1$}
    \caption{The series of resolutions done to $K_{n,t}^0$ in the proof of Theorem~\ref{Thm:tbKnts}. This allows calculation of $\kappa(K_{n,t-4}^4)$ from $\kappa$ of $K_{n,t}^0$, $K_{n,t-2}^0$, and $K_{n,t-2}^4$.}
    \label{fig:Resolutions1}
\end{figure}

\begin{figure}[t]
    \centering
   \includegraphics[scale=.55]{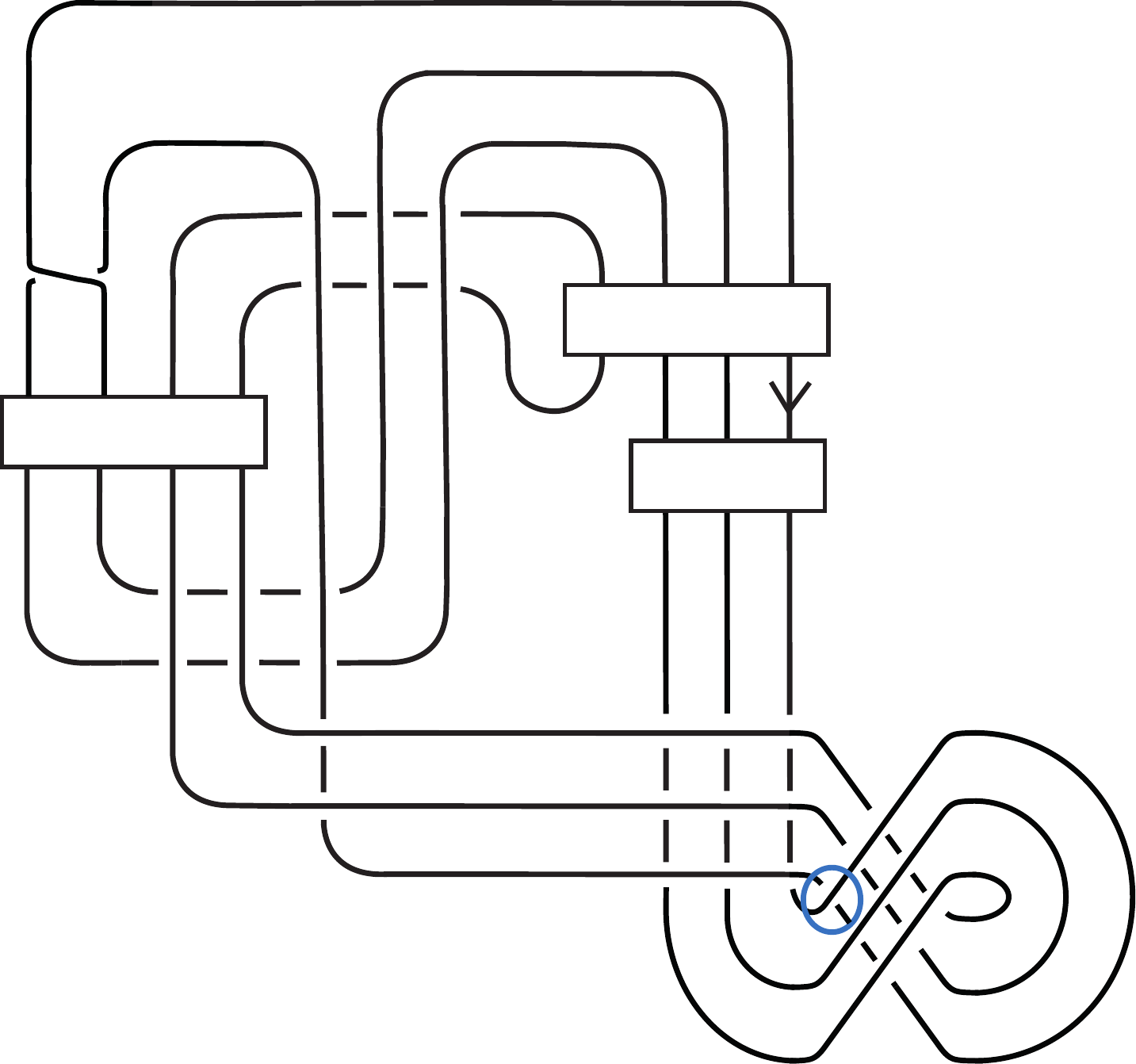}
\put(-203,124){$n$}
\put(-94,114){$t-4$}
\put(-99,145){$-2$}
    \caption{A diagram of the knot $K_{n,t}^0$.}
    \label{fig:Knt}
\end{figure}
\begin{figure}[t]
    \centering
   \includegraphics[scale=.55]{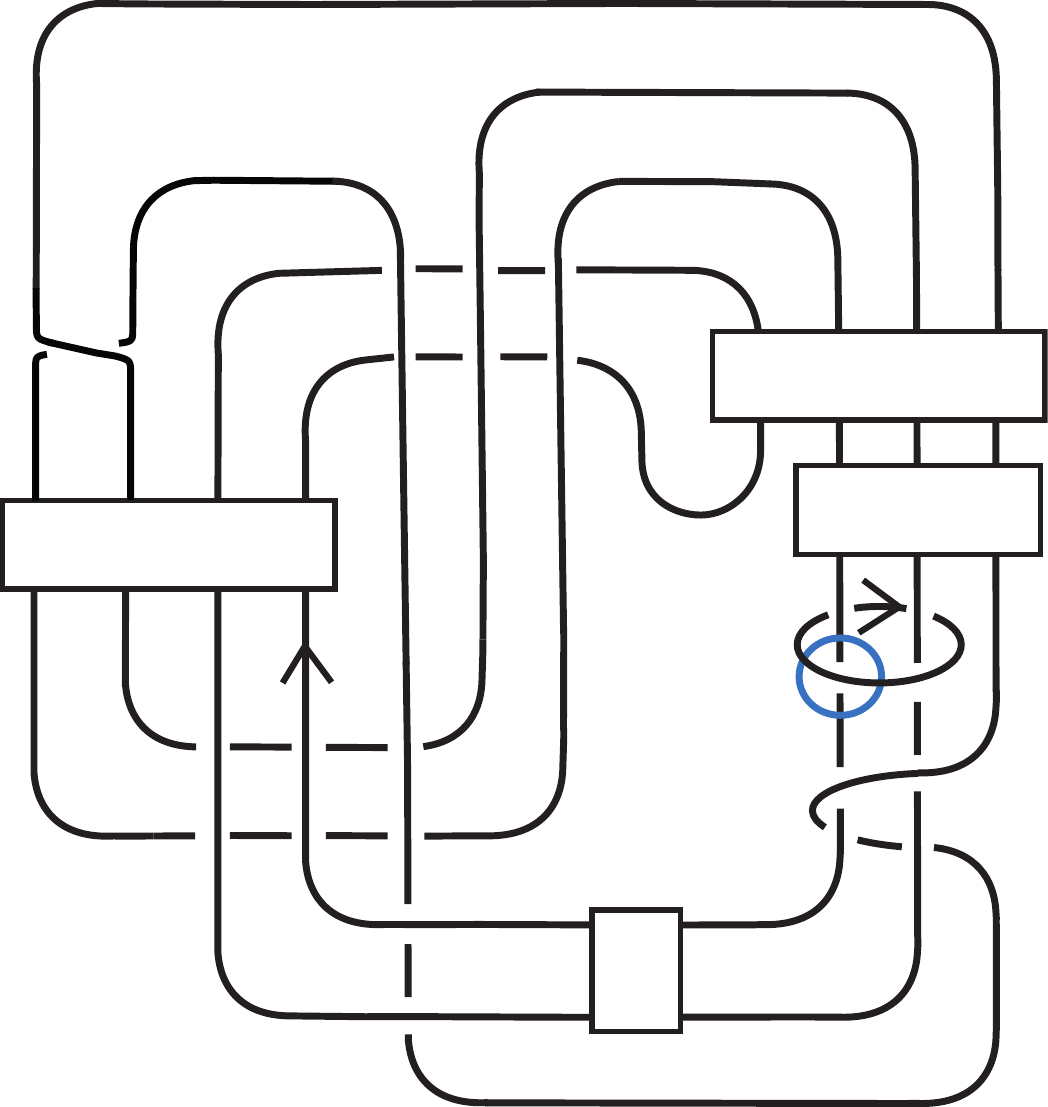}
\put(-143,86){$n$}
\put(-31,91){$t-4$}
\put(-37,113){$-2$}
\put(-68,18){$4$}
    \caption{A diagram of the knot $K_{n,t}^H$.}
    \label{fig:KntH}
\end{figure}


\begin{proof}[Proof of Theorem \ref{Thm:kappaKnts}]
We wish to prove that $\kappa(K_{n,t}^0)= -3-2n$, for all even integers $t$ with $t\le 2$. To do so, we use strong decreasing induction on $t$. The base case, $\kappa(K_{n,2}^0)$, is checked in  Lemma~\ref{Lem:kappaKn20}. Using the observation that $K_n \simeq K_{n,4}^0$, Lemma~\ref{Lem:kappaKn} shows that $\kappa(K_{n,4}^0)= -2-n \ge -3-2n$. 

For the inductive step, assume the result for $K_{n,t}^0$ and $K_{n,t-2}^0$, for some $t\le 4$, and we will prove it for $K_{n,t-4}^0$ (there will be a minor modification in the case $t=4$). In the exact sequence associated to the crossing of $K_{n,t}^0$ in Figure \ref{fig:Knt}, we have that $p=1$ and $q=-1$; as noted before we have $Res_1\simeq K_{n,t-4}^4$ and $Res_0\simeq K_{n,t}^H$. Then by the induction hypothesis we have $\kappa(K_{n,t}^0) = -2n-3$, and therefore for $u\le-2n-4$ we have
\begin{equation}\label{eq:KntH}
HKh^{u-1}(K_{n,t}^H)\cong HKh^{u-2}(K_{n,t-4}^4).
\end{equation}
Observe that in the case $t = 4$ we have instead $\kappa(K^0_{n,4}) = -n-2\geq -2n-3$, which suffices to infer \eqref{eq:KntH} for the same range of $u$, and more, in this case as well.\\

\noindent {\bf Claim.} $\kappa(K_{n,t}^H)= -2n-6.$
\begin{proof}[Proof of Claim.]
This time the exact sequence associated to the crossing indicated in Figure~\ref{fig:KntH} has $p=q=-1$. We have $Res_0\simeq K_{n,t-2}^4$ and $Res_1\simeq K_{n,t-2}^0$, and therefore the inductive hypothesis and the sequence in~\eqref{LES} give us that for $u\le-2n-4$,
\[
HKh^{u-1}(K_{n,t}^H)\cong HKh^{u-2}(K_{n,t-2}^4).
\]
By the inductive hypothesis and Lemma~\ref{Lem:changes} we have that $$\kappa(K_{n,t-2}^4)= -2n-7,$$
and the claim follows.
\end{proof}
The claim combined with Equation~\eqref{eq:KntH} imply that $\kappa (K_{n,t-4}^4)= -2n-7$. Lemma~\ref{Lem:changes} gives us that $\kappa(K_{n,t-4}^0)= -2n-3$, as desired. This concludes the proof of Theorem \ref{Thm:kappaKnts}, and hence also that of Theorem \ref{Thm:tbKnts}.
\end{proof}
\begin{figure}[b]
    \centering
   \includegraphics[width=4in]{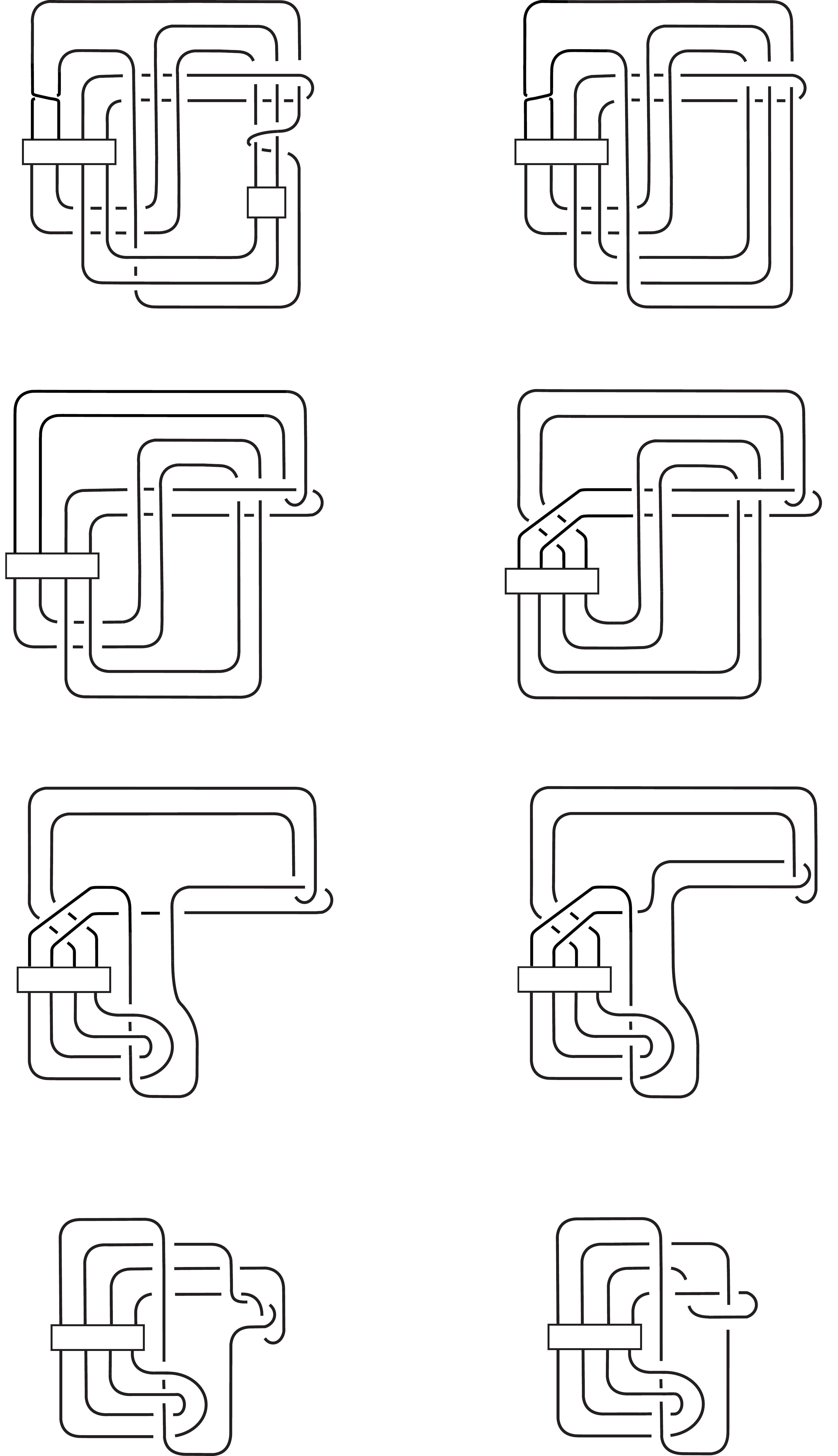}
\put(-268,447){$n$} 
\put(-200,429){{\small $2$}} 
\put(-97,447){$n$}   
\put(-275,304){$n$}   
\put(-102,299){$n$}   
\put(-270,161){$n$} 
\put(-98,161.5){$n$}
\put(-259,38){$n$} 
\put(-87,38.7){$n$} 
\put(-300,390){(a)} 
\put(-132,390){(b)} 
\put(-300,250){(c)} 
\put(-132,250){(d)} 
\put(-300,110){(e)} 
\put(-132,110){(f)} 
\put(-300,-13){(g)}
\put(-132,-13){(h)}  
\caption{The series of isotopies needed to go from $K_{n,4}^0$ to $K_n$. We leave it to the reader to check that $K_{n,4}^0$ is isotopic to the first picture (top left), also that the last picture (bottom right) is isotopic to $K_n$.}
    \label{Isotopy}
\end{figure}

\section{Proof of Theorem \ref{Thm:upperbd}} \label{Sec:upperbd}

The proof follows the general outline of Plamenevskaya's proof \cite{Olga2004} that for a Legendrian knot $\L$ in the standard contact 3-sphere with Thurston-Bennequin number $\tb(\L)$ and rotation number $\rot(\L)$, we have $\tb(\L) + \rot(\L)\leq 2\tau(\L) - 1$. The key point is that if $W: (Y_1,\xi_1)\to (Y_2,\xi_2)$ is a Stein cobordism, where $\xi_i$ are the contact structures induced by the Stein structure, then the induced homomorphism in Heegaard Floer homology $\hfhat(-Y_2)\to \hfhat(-Y_1)$ carries the contact invariant $c(\xi_2)$ to $c(\xi_1)$. More particularly, it is known that if $J$ is the given Stein structure on $W$, with associated $\spinc$ structure $\s_J$, then the map induced by $\s_J$ in $\hfhat$ has the stated property (this follows from \cite[Lemma 2.11]{Ghiggini1}, for example). If $Y_1$ is the standard 3-sphere then $c(\xi_{std})$ is nonzero, so in particular $(W,\s_J)$ induces a nontrivial homomorphism in Floer homology. Since the two homomorphism $\hfhat(Y_1)\to \hfhat(Y_2)$ and $\hfhat(-Y_2)\to \hfhat(-Y_1)$, induced by considering $W$ as a cobordism in each direction, are transposes, the map induced by $(W, \s)$ from $\hfhat(S^3)$ to $\hfhat(Y_2, \s|_{Y_2})$ is nontrivial when $\s = \s_J$ for a Stein structure $J$ on $W$.


In light of these remarks, inequality \eqref{weakbd} of Theorem \ref{Thm:upperbd} is a consequence of the following fact about homomorphisms induced by 2-handle additions. Here we always consider Floer homology groups with coefficients in the field $\F = \zee/2\zee$.

\begin{thm}\label{Thm:constraint} Let $W = W_n(K) : S^3 \to S^3_n(K)$ be the cobordism obtained by adding a 2-handle with framing $n$ along a knot $K\subset S^3$. If $\s$ is a $\spinc$ structure on $W$ such that $(W,\s)$ induces a nontrivial homomorphism $\hfhat(S^3)\to \hfhat(S^3_n(K))$, then
\[
|\langle c_1(\s), \Sigma\rangle | + n \leq 2\tau(K),
\]
where $[\Sigma]$ is a generator of $H_2(W;\zee)$.
\end{thm}

\begin{proof} If $n$ is sufficiently large, this follows from Proposition 3.1 of \cite{OSfourball}. For general $n$, recall that the homomorphism induced by $(W_n(K),\s)$ can be understood in terms of the chain complex computing knot Floer homology, by a recipe described by Ozsv\'ath and Szab\'o in \cite{OSintsurg}. Following this recipe, to a knot $K\subset S^3$ we associate a certain sequence of chain complexes $\{A_s\}_{s\in\zee}$, well-defined up to chain homotopy equivalence, to which we add the collection $\{B_s\}_{s\in \zee}$ where each $B_s$ is (chain homotopy equivalent to) the complex $\widehat{CF}(S^3)$. For each $s$ there are chain maps $v_s, h_s: A_s\to \cfhat(S^3)$ that we consider as homomorphisms $v_s: A_s\to B_s$ and $h_s: A_s\to B_{s+n}$. These complexes and chain maps enjoy various properties: 
\begin{itemize}
\item For each $s$, $A_s$ is quasi-isomorphic to $A_{-s}$.
\item For $s\gg 0$, $v_s$ is a quasi-isomorphism while $v_{-s}$ is zero.
\item The induced homomorphism $v_{s*}: H_*(A_s)\to H_*(B_s)$ is nontrivial if and only if the homomorphism $h_{-s*}: H_*(A_{-s})\to H_*(B_{-s+n})$ is nontrivial.
\item For $s\ll 0$, $h_s$ is a quasi-isomorphism while $h_{-s}$ is zero.
\end{itemize}

Now assemble the $A_s$ and $B_s$ into a chain complex $\X_n$ as follows. Define a chain map $D_n: \bigoplus A_s\to \bigoplus B_s$ by declaring the $s$-th entry of the image under $D_n$ of a vector $(\ldots, a_\ell,\ldots)$ to be the element $v_s(a_s) + h_{s-n}(a_{s-n})$. Then $\X_n$ is the mapping cone of $D_n$: 
its chain group is the direct sum of all $A_s$ and $B_s$ (for $s\in \zee$), while the differential is the sum of the differentials on each $A_s$ and $B_s$ and the chain map $D_n$.

The main results of \cite{OSintsurg} include the following: 
\begin{enumerate}
\item The homology of $\X_n$ is isomorphic to the Heegaard Floer homology $\hfhat(S^3_n(K))$, where $S^3_n(K)$ denotes the result of $n$-framed surgery along $K$.
\item If $W_n(K) : S^3\to S^3_n(K)$ is the trace of the surgery, and $\s$ is a \spinc structure on $W_n(K)$, then the homomorphism $\hfhat(S^3) \to \hfhat(S^3_n(K); \s)$ induced by $\s$ corresponds under the isomorphism above to the map in homology induced by the inclusion of the subcomplex $B_s = \cfhat(S^3)$ in $\X_n$, where $s\in \zee$ is characterized by the equation
\[
\langle c_1(\s), [\Sigma]\rangle + n = 2s.
\]
\end{enumerate}
Note that while the last equation depends on an orientation of $\Sigma$, a surface representing the generator of second homology of $W_n(K)$, this technicality is unimportant since Floer homology (and homomorphisms induced by cobordism) is invariant under replacement of a \spinc structure by its conjugate. Since this operation has the effect of replacing $c_1(\s)$ by its negative, there is no harm in fixing the sign of $[\Sigma]$ arbitrarily (strictly, the construction of the $A_s$ depends on an orientation of $K$, and this choice ultimately fixes all such signs). 

Combining the facts above, we see that to prove the theorem it suffices to constrain the values of $s$ for which the inclusion $B_s\to \X_n$ induces a nonzero map in homology: in particular, it will suffice to show that if $s > \tau(K)$, then the resulting map is trivial.

To understand this argument, it will be helpful to recall some of the structure of the complexes $A_s$ and $B_s$ and the maps between them. The constructions in \cite{OSknot} show that a knot $K\subset S^3$ gives rise to a bigrading on the chain group $CF^\infty(S^3)$, meaning a $\zee\oplus\zee$-valued function on the generators, with the property that the boundary operator is non-increasing in both gradings. In this context the endomorphism $U$ of $CF^\infty$ has bidegree $(-1,-1)$, while the subcomplex $CF^-$ is the span of those generators having bidegree $(i,j)$ with $i <0$. The complex $\cfhat$ is then a sub-quotient of $CF^\infty$ and corresponds to the span of those generators with $i = 0$; thus $j$ gives rise to a filtration on $\cfhat$, and the homology of the associated graded complex corresponding to a fixed value of $j$ is the {\it knot Floer homology} $\widehat{HFK}(S^3, K, j)$.

The invariant $\tau(K)$ is defined in terms of this filtration of $\cfhat$ as follows: if $\eff_s$ denotes the subcomplex of $\cfhat$ spanned by generators with bigrading $(0,j)$ for $j\leq s$, we let
\[
\tau(K) = \min \{s \,| \, \mbox{ inclusion induces a surjection $H_*(\eff_s)\to \hfhat(S^3) = \F$}\}.
\]
Note that $H_*(\eff_s) \to \hfhat(S^3)$ is surjective for any $s\geq \tau(K)$, by factoring the inclusion of $\eff_{\tau(K)}$ through $\eff_s$.

The complexes $A_s$ are also subquotients of $CFK^\infty$ (which is the notation for $CF^\infty$ when the latter is considered with the $\zee\oplus \zee$ bigrading as above): precisely, $A_s$ is spanned by those generators of $CFK^\infty$ in bigrading $(i, j)$, where $\max(i, j-s) = 0$. Thus, picturing the bigraded summands of $CFK^\infty$ as lying at lattice points in the $(i,j)$ plane, $A_s$ corresponds to the portion of the axis $i = 0$ at or below coordinate $s$, together with the horizontal strip at vertical coordinate $s$ and with nonpositive $i$-coordinate. Following \cite{OSknot}, we write sub-quotient complexes obtained in this manner using notation such as $A_s = C\{\max(i, j-s) = 0\}$. The differential in $A_s$ is induced from that of $CFK^\infty$, so in particular $A_s$ contains a subcomplex $C\{i = 0\mbox{ and } j \leq s-1\} = \eff_{s-1}$.

The chain maps $v_s$ and $h_s$ are defined as follows. First, $v_s: A_s\to B_s = \cfhat(S^3)$ is the natural quotient $A_s\to C\{i =0 \mbox{ and } j \leq s\} = \eff_s$, followed by the inclusion $\eff_s\to \cfhat(S^3)$. For $h_s$ we recall that there is a chain homotopy equivalence $C\{j  =0\}\to C\{i = 0\} = \cfhat(S^3)$, and that the action of $U$ on $CFK^\infty$ induces a chain isomorphism $C\{j = s\} \to C\{j = 0\}$ for any $s$. Then $h_s$ is the composition of the quotient $A_s\to C\{i\leq 0\mbox{ and } j = s\}\subset C\{j = s\}$ with these two quasi-isomorphisms. 

Consider these maps in the case $s \geq \tau(K) + 1$. By definition, the subcomplex $\eff_{s-1}$ of $A_s$ contains a cycle $x$ whose image under the inclusion $\eff_{s-1}\to \cfhat(S^3)$ generates $\hfhat(S^3)$. Therefore $v_{s*}([x])$ is the generator of $H_*(B_s) = \hfhat(S^3)$. On the other hand, since $x$ clearly lies in the kernel of the quotient $A_s \to C\{j = s\}$, we have that $h_{s*}([x]) = 0$. 

Thus, for any $s\geq \tau(K)+1$, the generator $[y_s]$ of the homology of $H_*(B_s)$ is the image of $[x]$ under the map induced by $D_n$, and in particular $[y_s] = 0$ in $H_*(\X_n)$. This proves that whenever $s\geq \tau(K)+1$ the inclusion $B_s\to \X_n$ is trivial in homology, as desired.

Finally, the absolute values appearing in the statement of the theorem may be added by the conjugation invariance of maps induced by cobordisms.
\end{proof}

We remark that under certain circumstances the proof above proves a little more. Namely, observe that $A_{\tau(K)}$ maps onto $\eff_{\tau(K)}$, and the latter contains a class generating $\hfhat(S^3)$. If this class can be represented by a cycle $x\in A_{\tau(K)}$ whose image under $h_{\tau(K)*}$ is trivial, then the same argument goes through to show that the inclusion of $B_s \to \X_n$ is trivial in homology for all $s\geq \tau(K)$.

Assume, then, that $v_{\tau(K)}$ is surjective in homology. Comparing with Hom \cite[section 2.2]{homsummand}, this assumption is equivalent to the statement that $\epsilon(K)$ is either 0 or 1 (here $\epsilon(K)\in\{1,0,-1\}$ is the concordance invariant defined by Hom in \cite{Hom1}). The assumption that $h_{\tau(K)*}([x]) = 0$ is then equivalent to saying that $v_{\tau(K)}$ and $h_{\tau(K)}$ induce distinct maps in homology: the ``only if'' part is clear; for ``if'' observe that if $h_{\tau(K)*}$ is not the zero map then we can replace $[x]$ by $[x] + c$ for some class $c\in \ker(v_{\tau(K)*}) \setminus \ker(h_{\tau(K)*})$.

Now recall Lemma 4.2 of \cite{marktosun}, which asserts that a knot $K\subset S^3$ has $\epsilon(K) = 0$ if and only if $v_{\tau(K)}$ and $h_{\tau(K)}$ induce the same nonzero map in homology. We conclude that if $\epsilon(K) = 1$ the two maps are different and the desired class $[x]$ exists. Therefore:

\begin{cor}[of the proof] \label{Cor:constraint} If $K\subset S^3$ is a knot with $\epsilon(K) = 1$, then for $\s$ a \spinc structure on $W_n(K)$ inducing a nontrivial map in homology we have 
\[
|\langle c_1(\s),\Sigma\rangle | + n \leq 2\tau(K) - 2.
\]
In particular for such $K$ we have
\[
\sf(K) + \cbar(K) \leq 2\tau(K) - 2.
\]
\end{cor}
This, together with Theorem \ref{Thm:constraint}, proves Theorem \ref{Thm:upperbd}.

\clearpage
\bibliographystyle{amsalpha2}
\bibliography{satellite}
\end{document}